\documentclass[10pt,a4paper,reqno]{amsart}

\usepackage{amssymb}

\usepackage{graphicx}

\usepackage[cmtip,all]{xy}

\usepackage{}

\newtheorem{theorem}{Theorem}[section]

\theoremstyle{definition}

\newtheorem{example}[theorem]{Example}
\theoremstyle{prop}
\newtheorem{prop}[theorem]{Proposition}
\theoremstyle{coro}
\newtheorem{coro}[theorem]{Corollary}

\theoremstyle{remark}

\numberwithin{equation}{section}

\begin{document}

\title{$p$-subgroups of units in $\mathbb{Z}G$}

\author{Wolfgang Kimmerle}
\address{ Fachbereich Mathematik, IGT,  Universit\"{a}t
Stuttgart, Pfaffenwaldring 57, 70550 Stuttgart, Germany}
\email{kimmerle@mathematik.uni-stuttgart.de}
\thanks{The authors were partially supported by the DFG priority program SPP 1489}
\author{Leo Margolis}
\address{Facultad de Matem\'aticas, Universidad de Murcia, 30100 Murcia, Spain}
\email{leo.margolis@um.es}
\date{February 15th, 2016}


\begin{abstract}
We consider the question whether a Sylow like theorem is valid in the normalized units 
of integral group rings of finite groups. After a short survey on the
known results we show that this is the case for integral group rings
of Frobenius groups. This completes work of M.A.~Dokuchaev,
S.O.~Juriaans and V.~Bovdi and M.~Hertweck. We analyze projective linear simple groups 
and show what can be achieved for $p$-subgroups with known methods.
\end{abstract}

\maketitle

\section{Introduction}

Let $G$ be a finite group. The integral group ring of $G$ is denoted by $\mathbb{Z} G$ 
and $\mathrm{V}(\mathbb{Z} G) $ denotes the subgroup of the unit group $\mathrm{U}(\mathbb{Z} G)$ consisting of 
all units with augmentation 1. 
\vskip1em
The question whether torsion subgroups of $\mathrm{V}(\mathbb{Z} G)$ are isomorphic (or
even conjugate) to subgroups of $G$ has a long history. By \cite{HertweckAnnals} we
know that in general the answer is negative, i.e. $\mathrm{V}(\mathbb{Z} G)$ may have 
torsion subgroups which are not isomorphic to subgroups of $G .$ The
smallest counterexample is that one constructed by M.~Hertweck and has
derived length four and order 
$2^{21} \cdot 97^{28} .$ However for many important classes of groups the
question is open. Thus the following question is in the focus of present 
research.  
\vskip1em

{\bf The subgroup isomorphim problem SIP.} Let $H$ be a given finite group
Suppose that whenever $H$ occurs as subgroup of $\mathrm{V}(\mathbb{Z} G)$ then $H$ is
isomorphic to a subgroup of $G .$ Then we say that the subgroup
isomorphism problem for $H$ has a positive answer.  
\vskip1em

Note that SIP contains the isomorphism problem IP for integral group rings (i.e. the question whether $\mathbb{Z} G \cong \mathbb{Z} H$ implies $G \cong H ).$ 
Thus if SIP has a positive answer for $H$, also the isomorphism problem
has a positive solution. Consequently classes of finite groups for
which IP is valid are of special interest for SIP.  
 
SIP is especially open for the case when $H$ is abelian or 
when $H$ is a $p$-group. In this article we
shall concentrate on the situation of $p$-subgroups.
SIP for $p$-groups leads naturally to the question of a 
Sylow or a Sylow like theorem in $\mathrm{V}(\mathbb{Z} G) .$ We say that a 
\vskip1em
{\bf Strong Sylow like theorem} is valid in $\mathbb{Z} G$, if for each
prime $p$ each finite $p$-subgroup of $\mathrm{V}(\mathbb{Z} G)$ is conjugate in $\mathbb{Q} G$ to a $p$-subgroup of $G $ 
\vskip1em
and we speak of a 
\vskip1em
{\bf Weak Sylow like theorem}, if for each prime $p$ each finite
$p$-subgroup of $\mathrm{V}(\mathbb{Z} G)$ is isomorphic to a $p$-subgroup of $G .$ 
\vskip1em

The word "Sylow like" is justified, because conjugacy takes place in
$\mathbb{Q} G.$ Although for abelian $G$ any $p$-subgroup of $\mathrm{V}(\mathbb{Z}G)$ is a subgroup of $G$, in general a finite $p$-subgroup of $\mathrm{V}(\mathbb{Z} G)$  is not conjugate within $\mathbb{Z} G$ to a subgroup of $G .$ Already in the integral group ring of the 
smallest nonabelian group $S_3$ there are different conjugacy classes of involutions in $\mathrm{V}(\mathbb{Z}
S_3) ,$ cf. \cite[p. 103]{KimDMV} or \cite[Example 3.4]{HertweckColloq}.  
  
Evidence that a Sylow like theorem for integral group rings is valid is given by the fact that in
$\operatorname{GL}(n, \mathbb{Z})$ a strong Sylow like theorem holds
\cite{AboldPlesken} and that it holds when $G$ is a $p$-group by the
celebrated results of Roggenkamp - Scott \cite{RogSco} and Weiss
\cite{WAnn}.
Indeed it is an open question, whether a strong Sylow like theorem holds
for each finite group $G .$ As a first goal however may serve a weak
Sylow like theorem which is obviously equivalent to a positive answer
to SIP for $p$-groups.      
If a Sylow like theorem (strong or weak) is established for a single
prime $q$ we say that it holds with respect to $q$.   

 In the first section we give a survey on known results concerning
 Sylow like theorems for integral group rings. With respect to solvable 
 groups many positive results are known. Thus insolvable and also
 simple groups are nowadays objects of investigations. In Section 2
 we complete the proof that a strong Sylow like theorem is valid
 provided $G$ is a Frobenius group. 
 Section 3 deals with simple linear
 groups. The so-called HeLP - method permits in many situations positive answers.
 However, we show that for $G = PSL(2,p^2), p \geq 7$ or $G = PSL(3,3)$ open questions
 show up. We construct explicit subgroups in $\mathbb{Q}G$. If these subgroups would be in $\mathbb{Z} G$ they would
 establish a counterexample to SIP for $p$-groups and to any kind of
 a Sylow like theorem. 
  
Finally in Section 4 we show that embedding of $\mathbb{Z} G$ into a suitable larger
 group ring $\mathbb{Z} H$ each unit of order $p$ is conjugate
 within $\mathbb{Q} H$ to an element of $G .$

\section{Known results}

Throughout we consider integral group rings of finite groups. The following classical results provide the basis of all investigations.

\begin{theorem}\label{CL} \cite[Corollary 4.1]{CohnLivingstone}, \cite{ZK}
Let $U$ be a finite subgroup of $\mathrm{V}(\mathbb{Z}G)$. Then the exponent of $U$ divides the exponent of $G$ and the order of $U$ divides the order of $G$.
\end{theorem}

Theorem \ref{CL} in particular establishes SIP for cyclic groups of prime power order. 

One of the most important results is the following of A.~Weiss. 

\begin{theorem} \cite[41.12]{SehgalBook} \label{Wei}
A strong Sylow like theorem holds in $\mathrm{V}(\mathbb{Z} G)$ with
respect
to the prime $p$ when $G$ has a normal Sylow $p$-subgroup.
\end{theorem} 

With respect to group extensions the situation is clear when the prime
$p$ does not divide the order of a normal subgroup. 

\begin{theorem} \cite[Theorem 2.2]{DokuchaevJuriaans} \label{ext}
Let $N$ be a normal subgroup of $G$ and let $U$ be a torsion
subgroup of $\mathrm{V}(\mathbb{Z} G)$ with $gcd(|U|,|N|) = 1 .$ 
 
Denote by $\kappa
:\mathrm{V}(\mathbb{Z} G) \longrightarrow \mathrm{V}(\mathbb{Z} G/N) $ the map on the units induced by the
reduction $G \longrightarrow G/N .$ 
Then $U$ is conjugate
to a subgroup of $G$ within $\mathbb{Q} G$ if, and only if, $\kappa (U)$ is
conjugate to a subgroup of $G/N$ within $\mathbb{Q} G/N .$
\end{theorem}
    
An immediate consequence of the two preceeding theorems is that a
 strong Sylow like theorem holds in $\mathbb{Z} G$ provided $G$ is a nilpotent - by -
nilpotent group \cite[Theorem 2.9]{DokuchaevJuriaans}. So in particular the case of
supersolvable groups is settled.    

\vskip1em

With respect to group bases, i.e. to torsion subgroups of $\mathrm{V}(\mathbb{Z} G)$ with
the same order as $G$, even more is known. Recall that a group is called $p$-constrained when the generalized Fitting subgroup $F^*(G/O_{p'}(G))$ is a $p$-group
(equivalently is to say that $G/O_{p'}(G)$ has a normal $p$ - subgroup containing its own centralizer, cf. \cite[Ch.X,15.5]{HB3}). 
In particular any $p$-solvable group is $p$-constrained. 

\begin{theorem} \label{Sylowsoluble} \cite{KimmerleRoggenkamp} 
\begin{itemize}
\item[a)] Let $G$ be a $p$-constrained group. Then Sylow $p$-subgroups of group bases are rationally conjugate.
\item[b)] Let $G$ be a finite solvable group and let $H$
  be a group basis of $\mathbb{Z} G$. Let $p$ be a prime. 
Then each $p$ -subgroup of $H$ is conjugate within $\mathbb{Q} G$ to  a
subgroup of a Sylow $p$-subgroup of $G .$   
\end{itemize}
\end{theorem}

Roggenkamp and Scott discovered that the isomorphism problem has a
strong positive solution provided the generalized Fitting subgroup $F^*(G)$ is
a $p$-group \cite[Theorem 19]{Rog}, \cite{Sco} rsp. Different group bases of integral 
group rings of such groups are even $p$-adically conjugate. This result
is often called the {\bf $F^*$-theorem}, see also \cite{HertweckUnits, HertweckKimmerle} for details. Clearly the $F^*$-theorem and
Theorem \ref{ext} establish the proof of Theorem \ref{Sylowsoluble}. 
In contrast to Theorem \ref{Wei}
the $F^*$-theorem does not make any statements on torsion $p$-subgroups which are not
contained in a group basis. The example given in \cite[XIV,\S2,2.1 Proposition]{RogTay} shows that
in the 2-adic group ring $\mathbb{Z}_2S_4$ exist two non conjugate subgroups
of order 8. So even  
under the assumptions of the $F^*$-theorem ($F^*(S_4) = F(S_4)$ is a
$2$-group) torsion $p$-subgroups of
$\mathrm{V}(\mathbb{Z} G)$ are in general not $p$-adically conjugate to a subgroup of $G .$ 

Whether a Sylow like theorem holds for all solvable groups is still an
open question. A 
minimal counterexample to this has by Theorem \ref{ext} the property that its
Fitting subgroup is a $p$-group. This explains why an extension of
the $F^*$-theorem to torsion subgroups is highly desired.  
\vskip1em       
     
If one assumes that Sylow $p$-subgroups of $G$ have a special
structure much more is known. This is especially the case when $G$ has
abelian or quaternion Sylow subgroups. We collect the known results in
one theorem.

\vskip1em

\begin{theorem} \label{hamiltonian}
\begin{itemize} 
\item[a)]  If $G$ has cyclic Sylow $p$-subgroups then with respect
  to $p$ the weak Sylow like 
  theorem holds for $\mathbb{Z} G.$ \cite{KimmerleC2C2, HertweckCpCp} 
\item[b)] If $G$ is $p$-constrained and $G$ has abelian Sylow $p$-subgroups, then with respect to $p$ the strong Sylow like 
  theorem holds for $\mathbb{Z} G$ \cite[Proposition 2.11]{DokuchaevJuriaans}, \cite[Proposition 3.2]{BaechleKimmerle}.  
\item[c)]  The weak Sylow like theorem holds with respect to $2$ provided $G$ has abelian or generalized quaternion Sylow $2$-subgroups \cite[Proposition 4.8]{KimmerleSylow} or dihedral Sylow $2$-subgroups \cite{SIP}.  
\item[d)]  The weak Sylow like theorem holds for an odd  prime $p$ in $\mathrm{V}(\mathbb{Z} G)$ provided Sylow $p$-subgroups of $G$ are isomorphic to $C_p \times C_p.$ This follows from Theorem \ref{CL}. 
\end{itemize}
\end{theorem}

Theorem \ref{hamiltonian} a) establishes SIP for groups of the form $C_p \times C_p$ with any prime $p$. The only other non-cyclic group for which SIP is known is $C_4 \times C_2$ \cite{SIP}.

Certainly further results with respect to general finite groups are missing.  
At least for small abelian Sylow $2$-subgroups it is known that a 
strong Sylow like
theorem holds.    
   
\begin{prop} \cite{SIP} \cite[Proposition 3.4]{BaechleKimmerle} Let $G$ be a group whose Sylow $2$-subgroups are of order $\leq 8$ such that $G$ is not isomorphic to the alternating group of degree $7$. Then each $2$-subgroup of  $V (\mathbb{Z} G)$ is rationally conjugate to a subgroup of $G .$  
\end{prop}

\section{Frobenius Groups}
\begin{example}\label{CoverS5}
Let $G$ be the covering group of the symmetric group of degree $5$
whose Sylow $2$-subgroup is a generalized quaternion group. We show
that a Strong Sylow like theorem holds for $\mathrm{V}(\mathbb{Z}G)$. 

The \texttt{GAP}-Id of $G$ is \texttt{[240, 89]}. The Sylow $2$-subgroup of $G$ is a generalized quaternion group of
order $16$ and the Sylow $3$- and $5$-subgroups of $G$ are
cyclic. Thus by Theorem \ref{hamiltonian} any $p$-subgroup of
$\mathrm{V}(\mathbb{Z}G)$ is isomorphic to a subgroup of $G$. 
By \cite[Example 1]{BovdiHertweck} the Zassenhaus Conjecture holds for
$G$. This implies that any $3$- or $5$-subgroup of
$\mathrm{V}(\mathbb{Z}G)$ is rationally conjugate to a subgroup of
$G$. 
So let $U \leq \mathrm{V}(\mathbb{Z}G)$ be a $2$-group and let $P$ be
a Sylow $2$-subgroup of $G$. If $U$ is cyclic it is rationally
conjugate to a subgroup of $P$ again by \cite[Example
1]{BovdiHertweck}. 

So assume $U$ is not cyclic, then $U$ is a
(generalized) quaternion group of order $8$ or $16$. The group $G$ has
exactly one maximal normal subgroup $N$, which is isomorphic to
$\operatorname{SL}(2,5)$, and contains exactly two conjugacy classes
of elements of order $4$. Let $a$ and $b$ be representatives of these
classes, such that $a$ lies in $N$ while $b$ does not.
Since $N$ has index $2$ in $G$ there is a one-dimensional
representation $\sigma$ of $G$ mapping $N$ onto $1$ and elements
outside of $N$ onto $-1$. Assume first that $U$ is of order $8$ and
let $U$ be generated by $u$ and $v$. Both generators are rationally
conjugate to elements of $G$. If one knows to which elements of $P$
the units $u$ and $v$ are conjugate one also knows to which elements
$u^iv^j$ are conjugate - this may be read of from the value of
$\sigma$. Let $c$ and $d$ be elements of $P$ such that $u$ is
rationally conjugate to $c$, $v$ is rationally conjugate to $d$ and
$\langle c, d \rangle$ is a quaternion group of order $8$, this
construction is always possible in $G$. 
Then the isomorphism $\varphi:U \rightarrow \langle c, d \rangle$
mapping $u$ to $c$ and $v$ to $d$ preserves the character values of
all irreducible complex characters and hence $U$ and $\langle c,d
\rangle$ are rationally conjugate by \cite[Lemma 4]{Valenti}. 
In the same manner one handles the case when $U$ has order $16$. Let $U = \langle u, v \rangle$ such that $u$ has order $8$ and $v$ has order $4$. Once we know to which group elements $u$ and $v$ are rationally conjugate $\sigma$ again provides this information for $u^iv^j$. This allows again to construct a character value preserving isomorphism between $U$ and $P$.
\end{example}

\begin{theorem}
Let $G$ be a Frobenius group. Then a strong Sylow like theorem holds
in $\mathbb{Z} G$.
\end{theorem}
\begin{proof}
By \cite[Theorem 6.1]{DokuchaevJuriaansMilies} we may assume that the
symmetric group of degree $5$ is an image of $G$ and we only have to
show that any $2$-subgroup $U \leq \mathrm{V}(\mathrm{Z}G)$ is
rationally conjugate to a subgroup of $G$. This implies that the
Frobenius kernel $N$ of $G$ is of odd order and by Theorem \ref{ext}
 we may proceed to $G/N$ which is isomorphic to a Frobenius complement
 $K$. 
By \cite[Theorem 18.6]{Passman} $K$ contains a normal subgroup $K_0$
of index $1$ or $2$ such that $K_0 \cong \operatorname{SL}(2,5) \times
M$ where $M$ is a characteristic normal subgroup of $K_0$ of odd
order. Again by Theorem \ref{ext} we may proceed to $K/M$ and this
group contains $\operatorname{SL}(2,5)$ as a normal subgroup of index
$1$ or $2$. If $K/M \cong \operatorname{SL}(2,5)$, the group $G$ can
not map onto $S_5$, thus $K/M$ is of order $240$. By \cite[Theorem
18.1]{Passman} the Sylow $2$-subgroup of $K/M$ is a generalized
quaternion group. 
Up to isomorphism there is only one non-solvable group of order $240$
whose 
Sylow $2$-subgroup is a generalized quaternion group - it is the
group handled in Example \ref{CoverS5}. Hence a strong Sylow like theorem holds for $K/M$ and then also for $G$. 
\end{proof}

\section{Crucial examples for simple linear groups}

In order to extend the known results from classes of more or less solvable groups to general finite groups it is a first step to consider nonabelian simple groups and their relatives. In this section we provide explicit examples where the known methods fail to prove a Sylow like theorem for $\mathrm{V}(\mathbb{Z}G)$. We first describe a technical ingredient.

Let $x^G$ be the conjugacy class of the group element $x$ in $G$ and let $u = \sum\limits_{g \in G} z_g g \in \mathbb{Z}G$. Then $\varepsilon_x(u) = \sum\limits_{g \in x^G} z_g$ is called the partial augmentation of $u$ with respect to $x$. Sometimes $\varepsilon_x(u)$ is also denoted as $\varepsilon_{x^G}(u)$. The relevance of partial augmentations for rational conjugation of units is provided by \cite[Theorem 2.5]{MarciniakRitterSehgalWeiss}: A unit $u \in \mathrm{V}(\mathbb{Z}G)$ is rationally conjugate to an element of $G$ if and only if $\varepsilon_x(u) \geq 0$ for all $x \in G$.  

We start with an example for the strong Sylow like theorem.
 
 \begin{theorem}\label{PSL2p2}
Let $G = \operatorname{PSL}(2, p^2)$, with $p$ prime.
\begin{itemize}
\item[a)]  If $p \leq 5$ and $U$ is a subgroup of $\mathrm{V}(\mathbb{Q}G)$ isomorphic to $C_p \times C_p$, then $U$ is rationally conjugate to a subgroup of $G$.
\item[b)] However, if $p \geq 7$, then $\mathrm{V}(\mathbb{Q}G)$ contains a subgroup $U'$ isomorphic to $C_p \times C_p$ such that all the elements of $U'$ are rationally conjugate to elements of $G$ while $U'$ is not rationally conjugate to a subgroup of $G$. 
\end{itemize}
\end{theorem}

\begin{coro}
Let $G = \operatorname{PSL}(2,p^2)$ and $p \leq 5$. Then a strong
Sylow like theorem holds in $\mathbb{Z}G$.
\end{coro}

\begin{proof}
For primes not equal to $p$ we can employ the same arguments as in the
proof of \cite[Theorem 2]{SylowPSL} and thus the strong Sylow like
theorem follows from the fact that elements of order $p$ of $\mathrm{V}(\mathbb{Z}G)$ are rationally conjugate to elements of $G$ \cite[Proposition 6.1]{HertweckBrauer} and Theorem \ref{PSL2p2}.
\end{proof}

\textit{Proof of Theorem \ref{PSL2p2}:}
Let $P$ be a Sylow $p$-subgroup of $G$ and let $U \leq \mathrm{V}(\mathbb{Z}G)$ such that $U \cong C_p \times C_p$. By \cite[Proposition 6.1]{HertweckBrauer} for every $u \in U$ there exists a $g \in P$ such that $\varepsilon_C(u) = \varepsilon_C(g)$ for all conjugacy classes $C$ of $G$. In case $p = 2$ there is only one conjugacy class of involutions in $G$ and thus any $u \in U$ is rationally conjugate to any $g \in P$. In particular any isomorphism between $U$ and $P$ fixes the character values for all complex irreducible representations of $G$. Hence $U$ and $P$ are rationally conjugate because the character determines an ordinary representation up to equivalence, see also \cite[Lemma 4]{Valenti}. 

So assume $p \geq 3$. There are two conjugacy classes of elements of order $p$ in $G$, let $c$ and $d$ be representatives of these classes. Note that any non-trivial power of an $p$-element is conjugate to the element itself. This may be seen e.g. since the values of all irreducible characters are rational for $c$ and $d$. In a first step we prove.

\textit{Claim:} There are exactly $\frac{p^2-1}{2}$ elements of $U$ which are rationally conjugate to $c$ and $\frac{p^2-1}{2}$ which are rationally conjugate to $d$.

We will apply the non-cyclic HeLP-method as described in \cite{PSL2p3} with the character $\eta$ given in Table \ref{eta}.

\begin{table}[h]
\centering
\begin{tabular}{r|ccc} 
 & $1$ & $c$ & $d$ \\ \hline  \\[-2ex]
$\eta$ & $ \frac{p^2-1}{2} $ & $ \frac{p+1}{2} $ & $\frac{-p+1}{2}$
\end{tabular}
 \medskip
 \caption{A character of $\operatorname{PSL}(2,p^2)$.}\label{eta}
\end{table}

Let $u \in U$ be a fixed element and let $x$ be the number of cyclic subgroups of $U$ rationally conjugate to $\langle u \rangle$. Then there are $p + 1 - x$ cyclic subgroups of $U$ not being rationally conjugate to $\langle u \rangle$. Let $\eta(v) \in \{ \eta(c), \eta(d) \}$ such that $\eta(v) \neq \eta(u)$. Let $\chi$ be a non-trivial character of $U$ such that ${\rm{ker}}(\chi) = \langle u \rangle$. Then 
\[\sum\limits_{\alpha \in \langle u \rangle} \chi(\alpha)\eta(\alpha) = \frac{p^2+1}{2} + (p-1)\eta(u)\]
and $\sum\limits_{\alpha \in \langle t \rangle \setminus 1} \chi(\alpha) = -1$ for any $t \in U \setminus \langle u \rangle$. Hence
\begin{align*}
\langle \eta, \chi \rangle_U &= \frac{1}{p^2}\sum\limits_{\alpha \in U} \chi(\alpha)\eta(\alpha) \\
&= \frac{1}{p^2}\left(\frac{p^2+1}{2} + (p-1)\eta(u) - (x-1)\eta(u) - (p+1-x)\eta(v)  \right)  \\
&= \frac{1}{p^2}\left( \frac{p^2+1}{2} + p\eta(u) + x(\eta(v) - \eta(u)) - (p+1)\eta(v) \right) \\
&= \frac{1}{p^2}\left\{ \begin{array}{ll} p^2 + \frac{p^2+p}{2} - xp, & \eta(u) = \eta(c) \\  \frac{-p^2-p}{2} + xp, & \eta(u) = \eta(d) \end{array}\right.
\end{align*}
Since $\langle \eta, \chi \rangle_U$ is a non-negative integer and $0 \leq x \leq p+1$ this implies $x = \frac{p+1}{2}$ and the claim is proven.

So let $u,v \in U$ such that $u$ is rationally conjugate to $c$ and $v$ is rationally conjugate to $d$. Then $U$ is rationally conjugate to $P$ if and only if there exist $g, h \in P$ such that $g$ is conjugate to $c$, $h$ is conjugate to $d$ and $uv^i$ is rationally conjugate to $gh^i$ for any $1 \leq i \leq p-1$ by \cite[Lemma 4]{Valenti}. Let $I$ be a subset of $\{1,\ldots,p-1\}$ such that $uv^i$ is rationally conjugate to $u$ if and only if $i \in I$. By the above claim $I$ contains exactly $\frac{p-1}{2}$ elements. If $p = 3$ or $p = 5$, then for any $I$ we can find suitable $g, h \in P$ and $U$ is thus rationally conjugate to $P$. This proves part a) of the theorem.

We will show next that for any choice of $I$ there is a $U' \leq \mathrm{V}(\mathbb{Z}G)$, such that $\langle u\rangle \times \langle v \rangle =  U' \cong C_p \times C_p$, realizing that choice and this will imply part b) of the theorem, since for $p = 7$ the choice $I = \{1,2,4\}$ does not correspond to any choice of $g,h \in P$ and for $p \geq 11$ this already follows from combinatorics since we can fix $g$ without loss of generality and there are then $\frac{p^2-1}{2}$ choices for $h$, but $\binom{p-1}{\frac{p-1}{2}} > \frac{p^2 - 1}{2}$, if $p \geq 11$. So let $I = \{i_1,\ldots,i_{\frac{p-1}{2}}\}$.

There are exactly two irreducible complex characters of $G$ which do not take the same values on $c$ and $d$. Apart from $\eta$ given in Table \ref{eta} this is a character $\tilde{\eta}$, where $\tilde{\eta}(c) = \eta(d)$ and $\tilde{\eta}(d) = \eta(c)$. In any other Wedderburn component of $\mathbb{Q}G$ the groups $P$ and $U'$ are then conjugate. Since $\eta$ and $\tilde{\eta}$ take different values on elements of $U'$, which would otherwise contradict that any element of $U'$ is rationally conjugate to an element of $P$, it suffices to construct $u$ and $v$ in the Wedderburn component corresponding to $\eta$. For that let $A$ be any rational $(p-1)\times(p-1)$-matrix of order $p$, e.g. the normal rational form of such a matrix. Denote by $E_{p-1}$ the identity matrix of size $p-1$. Writing in block form set
\[u = \left(1, E_{p-1}, A^{-i_1},\ldots,A^{-i_{\frac{p-1}{2}}}\right) \]
and
\[v = \left(1,A,\ldots,A\right)\]
where $A$ appears exactly $\frac{p+1}{2}$ times in $v$. Then $U'$ satisfies all assumptions and realizes the given $I$. \hfill $\qed$ \\

We proceed with an example where the known methods fail to prove a weak Sylow like theorem. Recall that the Sylow $3$-subgroup if $\operatorname{PSL}(3,3)$ is a non-abelian group of order $27$ with exponent $3$.

\begin{prop}\label{PSL33}
Let $G = \operatorname{PSL}(3,3)$. 
\begin{itemize}
\item[a)] If the Zassenhaus Conjecture holds for $G$, then a weak Sylow like theorem holds in $\mathrm{V}(\mathbb{Z}G)$.
\item[b)] There exists a subgroup $U \leq \mathrm{V}(\mathbb{Q}G)$ such that $U$ is isomorphic to an elementary-abelian $3$-group of order $27$ and every element of $U$ has integral partial augmentations at all conjugacy classes of $G$ of order $3$ and vanishing partial augmentations otherwise. 
\end{itemize}
\end{prop}

\begin{proof}
The Sylow $3$-subgroup of $G$ is a non-abelian group of order $27$ and exponent $3$.  There are two conjugacy classes of elements of order $3$ in $G$. Let $a$ and $b$ be representatives of these classes such that $a$ is central in some Sylow $3$-subgroup of $G$. Note that the character values of $a$ and $b$ are all rational, implying that any unit $u \in \mathrm{V}(\mathbb{Q}G)$ whose partial augmentations vanish on conjugacy classes not of order $3$ is rationally conjugate to its inverse. 

Assume first that the Zassenhaus Conjecture holds for $G$. Since up to isomorphism there is only one non-abelian group of order $27$ and exponent $3$ to prove part a), by Theorem \ref{CL}, it is sufficient to show that $\mathrm{V}(\mathbb{Z}G)$ contains no elementary-abelian group $U$ of order $27$. Let $x$ be the number of cyclic subgroups of $U$ whose non-trivial elements are rationally conjugate to $a$ and let $y$ be the number of cyclic subgroups of $U$ whose non-trivial elements are rationally conjugate to $b$. Then $0 \leq x,y \leq 13$ and $x + y = 13$. Let $\chi$ be a complex irreducible $16$-dimensional character of $G$, then $\chi(a) = -2$ and $\chi(b) = 1$. It follows that
\[ \frac{1}{27}\sum_{u \in U} \chi(u) = \frac{1}{27}(16 -4x +2y) = \frac{1}{27}(40 -6x) \]
has to be an integer. But this is not the case for any $x$. So in this case a weak Sylow like theorem holds for $\mathrm{V}(\mathbb{Z}G)$ by \cite[Proposition 5.1]{BaechleKimmerle}.  

We proceed by constructing an elementary-abelian group $U$ of order $27$ in $\mathrm{V}(\mathbb{Q}G)$ such that $U = \langle \alpha \rangle \times \langle \beta \rangle \times \langle \gamma \rangle$ and the elements of $U$ have the following partial augmentations. The partial augmentations on classes of order different from $3$ vanish for all non-trivial $u \in U$, moreover $(\varepsilon_a(\alpha), \varepsilon_b(\alpha)) = (3,-2)$ and $(\varepsilon_a(x), \varepsilon_b(x)) \in \{(1,0), (0,1)\}$ for any other non-trivial $u \in U$ such that 
\[(\varepsilon_a(\alpha^i\beta^j\gamma^k), \varepsilon_b(\alpha^i\beta^j\gamma^k)) = (1,0) \Longleftrightarrow i + j + k \equiv 0 \mod 3.\]

Restricting our attention to the degree of a character and its values on $a$ and $b$ any irreducible complex character of $G$ may be constructed as a linear combination with non-negative integer coefficients of characters being equal on $a$ and $b$ and the two characters given in Table \ref{CharPSL33}.

\begin{table}[h]
\centering
\begin{tabular}{r|ccc} 
 & $1$ & $a$ & $b$ \\ \hline  \\[-2ex]
$\chi$ & $12$ & $3$ & $0$ \\
$\varphi$ & $16$ & $-2$ & $1$
\end{tabular}
 \medskip
 \caption{Two characters of $\operatorname{PSL}(3,3)$.}\label{CharPSL33}
\end{table}

It thus suffices to construct $U$ in the two blocks of $\mathbb{Q}G$ corresponding to these two characters. For that let $A$ be some rational $2\times 2$-matrix of order $3$, e.g. $\begin{pmatrix} 0 & -1 \\ 1 & -1 \end{pmatrix}$, and denote by $E$ a $2\times 2$ identity matrix. In the block corresponding to $\chi$ set, writing in block matrix form,
\begin{align*}
\alpha &= (E, E, E, E, E, A), \\
\beta &= (E, E, A, A, A, A), \\
\gamma &= (E, A, A, E, A^2, A).
\end{align*}
In the block corresponding to $\varphi$ set
\begin{align*}
\alpha &= (A, A, A, A, A, A, A, A), \\
\beta &= (E, E, E, A, A, A^2, A^2, A^2), \\
\gamma &= (E, A, A^2, E, A^2, E, A, A^2).
\end{align*}
This construction realizes the partial augmentations on the elements of $U$ as given above.
\end{proof}

\section{Conjugacy in larger group rings}

In \cite[Problem 22]{Ari} the first author raised the following
problem: Given a unit $u \in \mathrm{V}(\mathbb{Z}G)$, does there
exist a group $H$ containing $G$ such that $u$ is conjugate to an
element of $G$ within $\mathrm{V}(\mathbb{Q}H)$? 

This may be regarded as a Sylow like theorem "between weak and strong".  M.~Hertweck remarked that this is true under the assumption that $u$ is of prime order, but never gave a proof. We embed this into a wider context and obtain Hertweck's remark as a corollary.

\begin{theorem}\label{Hall}
Let $G$ be a finite group, $p$ a prime and $P$ a $p$-subgroup of $\mathrm{V}(\mathbb{Z}G)$ of exponent $p$ isomorphic to some subgroup $U$ of $G.$ Then there exists a finite group $H$ containing $G$ such that $P$ is conjugate to $U$ within $\mathbb{Q}H.$
\end{theorem}

\begin{coro}
Let $G$ be a finite group and $p$ a prime. 
\begin{itemize}
\item[a)] Let $u \in \mathrm{V}(\mathbb{Z}G)$ be a unit of order $p.$ Then
there exists a finite group $H$ containing $G$ such that $u$ is
conjugate to some $g \in G$ within $\mathbb{Q}H.$ 
\item[b)] Let $U$ be a subgroup of $\mathrm{V}(\mathbb{Z} G)$ isomorphic to $C_p \times C_p .$
Then there exists a finite group $H$ containing $G$ such that $U$ is
conjugate to a subgroup of $G$ within $\mathbb{Q}H.$  
\end{itemize}
\end{coro}

{\bf Remark.} The corollary applies especially to the groups
$PSL(2,p^2)$ considered in Theorem \ref{PSL2p2}. Indeed in $\mathrm{V}(\mathbb{Z} PGL(2,p^2))$ a strong Sylow like theorem holds with respect to $p$.  \\

\textit{Proof of Theorem \ref{Hall}:} The proof will use Hall's Universal Group introduced in \cite{Hall} and nicely described in \cite[Chapter 6]{KegelWehrfritz}.

Let $c_1, \ c_2, \ldots ,c_k$ be representatives of the conjugacy classes
of elements of order $p$ in $G.$ If $c$ is an element in an other
conjugacy class in $G$ and $u$ an element in $P,$ then
$\varepsilon_c(u) = 0$ by \cite[Lemma 2.8]{HertweckColloq}. Let
$\Gamma$ be Hall's Universal Group containing $G.$ Then by
\cite[Theorem 6.1d)]{KegelWehrfritz} there exist elements
$\gamma_2, \ldots ,\gamma_k$ in $\Gamma$ such that $c_1^{\gamma_i} = c_i.$
Let $H = \langle G, \gamma_2, \ldots ,\gamma_k \rangle.$ 
Since $\Gamma$ is locally finite, $H$ is finite and $c_1, \ldots ,c_k$ are
all conjugate in $H.$ 

Thus viewing $P$ as a subgroup of $\mathrm{V}(\mathbb{Z}H)$ for every $u \in P$ we have $\varepsilon_c(u) = 1$, if $c$ is conjugate (within $H$) to $c_1$, and $\varepsilon_c(u) = 0$ otherwise. The same holds for any element $u \in U$. Hence for any isomorphism $\sigma: P \rightarrow U$ and any irreducible character $\chi$ of $H$ we get $\chi(\sigma(u)) = \chi(u)$ for every $u \in P.$ So by \cite[Lemma 4]{Valenti} $P$ is conjugate to $U$ within $\mathbb{Q}H.$  
\bibliographystyle{amsalpha}

\bibliography{Passman.bib}

\end{document}